\documentclass{article}
\usepackage[utf8]{inputenc}
\usepackage{fullpage}
\usepackage{amsmath}
\makeatletter
\renewcommand*\env@matrix[1][\arraystretch]{%
	\edef\arraystretch{#1}%
	\hskip -\arraycolsep
	\let\@ifnextchar\new@ifnextchar
	\array{*\c@MaxMatrixCols c}}
\makeatother

\usepackage{amssymb,amsfonts, amsthm}
\def\acts{\curvearrowright}
\usepackage[all,arc]{xy}
\usepackage{mathrsfs}
\usepackage{hyperref}
\usepackage{cleveref}
\usepackage{enumerate}
\usepackage{mathtools}
\usepackage{tikz-cd}
\usepackage{tikz}
\usepackage{sseq}
\usepackage{tabu}
\usepackage[style=alphabetic]{biblatex}
\def\zz{\mathbb{Z}}

\def\cc{\mathbb{C}}

\def\pp{\mathbb{P}}

\newcommand{\homoldeg}[3]{H_{#1}\left(#2; #3\right)}

\newcommand{\cohomoldeg}[3]{H^{#1}\left(#2; #3\right)}

\newcommand{\Par}[1]{\left(#1\right)}
\newcommand{\BPar}[1]{\left[#1\right]}
\newcommand{\WBPar}[1]{\left\{#1\right\}}
\newcommand{\Abs}[1]{\left|#1\right|}
\newcommand{\Bracket}[1]{\left\langle #1\right\rangle}

\theoremstyle{plain}
\newtheorem{thm}{Theorem}[section]

\theoremstyle{definition}
\newtheorem{cor}[thm]{Corollary}
\newtheorem{prop}[thm]{Proposition}

\newtheorem{rmk}[thm]{Remark}

\title{Monodromy of the family of cubic surfaces branching over smooth cubic curves}
\author{Ad\'an Medrano Mart\'in del Campo}
\date{\today}

\begin{document}

\maketitle

\begin{abstract}
	Consider the family of smooth cubic surfaces which can be realized as threefold-branched covers of $\mathbb{P}^{2}$, with branch locus equal to a smooth cubic curve. This family is parametrized by the space $\mathcal{U}_{3}$ of smooth cubic curves in $\mathbb{P}^{2}$ and each surface is equipped with a $\mathbb{Z}/3\mathbb{Z}$ deck group action.
	
	We compute the image of the monodromy map $\rho$ induced by the action of $\pi_{1}\left(\mathcal{U}_{3}\right)$ on the $27$ lines contained on the cubic surfaces of this family. Due to a classical result, this image is contained in the Weyl group $W\left(E_{6}\right)$. Our main result is that $\rho$ is surjective onto the centralizer of the image a of a generator of the deck group. Our proof is mainly computational, and relies on the relation between the $9$ inflection points in a cubic curve and the $27$ lines contained in the cubic surface branching over it.
\end{abstract}

\section{Introduction}\label{Intro}

Families of algebraic varieties are ubiquitous in algebraic geometry. A basic but often difficult question that arises for any such family is to determine its monodromy. A classical example is the \emph{universal cubic surface} given by the smooth fiber bundle
\begin{center}
	\begin{tikzcd}
		S\arrow[r, hookrightarrow] & \mathcal{E}_{3, 3}\arrow[d] & \WBPar{\Par{S, p}\mid p\in S}\arrow[d, mapsto] \\
		& \mathcal{U}_{3, 3} & S
	\end{tikzcd}
\end{center}
where $\mathcal{U}_{3, 3}$ is the parameter space of smooth cubic surfaces in $\pp^{3}$. This bundle induces a monodromy homomorphism $\rho: \pi_{1}\Par{\mathcal{U}_{3, 3}}\to \text{Aut}\Par{\cohomoldeg{2}{S}{\zz}}$. Klein and Jordan proved that the Galois group of the equation for the $27$ lines on a cubic is precisely $W\Par{E_{6}}$, the Weyl group of $E_{6}$, which coincides with the image of the monodromy homomorphism induced by the universal family $\mathcal{E}_{3, 3}\to \mathcal{U}_{3, 3}$. A natural way to further study this family is to study subfamilies of it. In this paper we study the subfamily of smooth cubic surfaces which can be realized as cyclic branched covers of $\pp^{2}$, branching over a smooth cubic plane curve. To do this we exploit connections between datum associated to cubic surfaces, and its analogue for cubic curves. The most notable example in this paper is the relation between the $27$ lines contained in a cubic surface, and the $9$ inflection points of the curve over which it branches.

The parameter space of homogeneous degree $d$ polynomials  in variables $x, y, z$ is given by
\[\pp\Par{\text{Sym}^{d}\Par{\cc^{3}}}=\pp^{N\Par{d}}\quad\text{where }N\Par{d}=\binom{d+2}{2}-1.\]
The \emph{vanishing locus} of $f\in \pp^{N\Par{d}}$ is defined as the set $V\Par{f}=\WBPar{P\in \pp^{2}\mid f\Par{P}=0}$. The \emph{discriminant locus} is the subset $\Delta_{d}\subset \pp^{N\Par{d}}$ consisting of polynomials whose vanishing locus is singular. The parameter space of \emph{smooth} degree $d$ plane curves in $\pp^{2}$ is therefore defined as
\[\mathcal{U}_{d}=\pp^{N\Par{d}}\setminus \Delta_{d}.\]

Let $f\in \mathcal{U}_{d}$ and let $C=V\Par{f}$ be its vanishing locus. Then $\BPar{C}=d \BPar{H}\in \homoldeg{2}{\pp^{2}}{\zz}$, where $\BPar{H}$ is the hyperplane class in $\pp^{2}$. The curve $C$ is a complex codimension $1$ submanifold of $\pp^{2}$, so there exists a cyclic $k$-fold branched cover $X$ of $\pp^{2}$ with branched locus equal to $C$ if and only if $\BPar{C}$ is a multiple of $k$ of $\BPar{H}$ (see \cite{morita}, Proposition 4.10) and this is equivalent to $k\mid d$. It is a classsical result (see \cite{zariski}) that the fundamental group of the complement of a smooth degree $d$ curve in $\pp^{2}$ is cyclic of order $d$. In light of this, for $f\in \mathcal{U}_{3}$ consider the cyclic $3$-fold branched cover
\begin{center}
	\begin{tikzcd}
			X_{f} \arrow[d, "p", "\zz/3\zz"']  & \\
			\pp^{2}\arrow[r, hookleftarrow] & V\Par{f} 
	\end{tikzcd}
\end{center}
The surface $X_{f}$ can be embedded into $\pp^{3}$ as a cubic surface $V\Par{w^{3}-f}$. In particular, $\mathcal{U}_{3}$ parametrizes all such surfaces of the form $V\Par{w^{3}-f}$. We define the \emph{universal $3$-branched cover of $\pp^{2}$} as the fiber bundle
\begin{align*}
\mathcal{E}_{3}=\WBPar{\Par{P, f}\in\pp^{3}\times \mathcal{U}_{3}\mid P\in V\Par{w^{3}-f}}&\to \mathcal{U}_{3} \\
\Par{P, f}&\mapsto f
\end{align*}
where the fiber of $\mathcal{E}_{3}\to \mathcal{U}_{3}$ is diffeomorphic to a smooth cubic surface. For the sake of simplicity in calulations that will be carried on further in this paper, choose the curve $f=y^{2}z-x^{3}+xz^{2}$ as a base point in $\mathcal{U}_{3}$. The action of $\pi_{1}\Par{\mathcal{U}_{3}}$ on $\cohomoldeg{2}{V\Par{w^{3}-f}}{\zz}$ induces a monodromy homomorphism
\[\rho:\pi_{1}\Par{\mathcal{U}_{3}}\to \text{Aut}\Par{\cohomoldeg{2}{V\Par{w^{3}-f}}{\zz}}.\]
In \cite{dolglib}, Dolgachev-Libgober describe $\pi_{1}\Par{\mathcal{U}_{3}}$ as a central extension
\begin{align}\label{extension}
1\to \mathcal{H}_{3}\Par{\zz/3\zz}\to\pi_{1}\Par{\mathcal{U}_{3}}\to SL_{2}\Par{\zz}\to 1
\end{align}
of $SL_{2}\Par{\zz}$ by the \emph{$\zz/3\zz$-points of the 3-dimensional Heissenberg group}, defined as
\[\mathcal{H}_{3}\Par{\zz/3\zz}=\WBPar{\begin{pmatrix}
	1 & a & c \\
	0 & 1 & b\\
	0 & 0 & 1
	\end{pmatrix}\mid a, b, c\in \zz/3\zz}.\]
The extension (\ref{extension}) is split, so $\pi_{1}\Par{\mathcal{U}_{3}}$ has a semidirect product structure $\mathcal{H}_{3}\Par{\zz/3\zz}\rtimes_{\varphi} SL_{2}\Par{\zz}$, which we describe later in the paper.

The image of $\rho$ can be further restricted by noting that the intersection form in $\cohomoldeg{2}{V\Par{w^{3}-f}}{\zz}$ remains invariant under the action of the image of $\rho$. As explained in \Cref{FactsCubics}, this implies that the image of $\rho$ is contained in the automorphism group of the $27$ lines contained in $V\Par{w^{3}-f}$ due to an argument in \cite{harris}. On the same paper, Harris shows this group is precisely $W\Par{E_{6}}$, the Weyl group of $E_{6}$. Moreover, each fiber of $\mathcal{E}_{3}$ has a $\zz/3\zz$ deck group action induced by its cyclic branched cover structure. A generator $T$ of this deck group action is given by
\begin{align*}
T:V\Par{w^{3}-f}&\to V\Par{w^{3}-f} \\
\BPar{x:y:z:w}&\mapsto \BPar{x:y:z:e^{-2\pi i/3}w}
\end{align*}
and $T$ induces an action $\Omega$ on $\cohomoldeg{2}{V\Par{w^{3}-f}}{\zz}$ which commutes with the image of $\rho$. As we shall see, $\Omega$ is realized as the image of a generator of the center of $\mathcal{H}_{3}\Par{\zz/3\zz}$ inside $\pi_{1}\Par{\mathcal{U}_{3}}$. Altogether, this shows that one may restrict the monodromy homomorphism $\rho$ to the centralizer $C_{W(E_{6})}(\Omega)$ of $\Omega$ in $W(E_{6})$, giving:
\[\rho:\pi_{1}\Par{\mathcal{U}_{3}}\to C_{W\Par{E_{6}}}\Par{\Omega}.\]
Our main result in this paper is the following.

\begin{thm}\label{MainTheorem}
	The monodromy representation
	\[\rho:\pi_{1}\Par{\mathcal{U}_{3}}\to C_{W\Par{E_{6}}}\Par{\Omega}\]
	of the universal $3$-branched cover of $\pp^{2}$ is surjective and its image is isomorphic to the semidirect product
	\[\mathcal{H}_{3}\Par{\zz/3\zz}\rtimes_{\varphi} SL_{2}\Par{\zz/3\zz}.\]
\end{thm}

The first step towards the proof of \Cref{MainTheorem} is to find geometric representatives of a basis for $\cohomoldeg{2}{V\Par{w^{3}-f}}{\zz}$ in terms of $V\Par{f}\subset \pp^{2}$. The crucial observation is that to each of the $9$ inflection points of $V\Par{f}$, we can associate $3$ lines in $V\Par{w^{3}-f}$, all of which lie over the given inflection point.

Then we reduce the range of $\rho$ to $W\Par{E_{6}}$. The action $\Omega$ induced by the $\zz/3\zz$ deck group action gives a permutation of the $27$ lines in $V\Par{w^{3}-f}$ which permutes the line triples lying over a same inflection point disjointly. The image of $\rho$ commutes with $\Omega$ and thus the image of $\rho$ is contained in $C_{W\Par{E_{6}}}\Par{\Omega}$.

Using the semidirect product structure on $\pi_{1}\Par{\mathcal{U}_{3}}$, we choose $4$ of its elements and compute their action on the geometric datum associated to the chosen basepoint curve \[f=y^{2}z-x^{3}+xz^{2}.\]
The key step of the proof is performing these computations explicitly. The group generated by the images of the $4$ chosen elements is then shown to be isomorphic to $\mathcal{H}_{3}\Par{\zz/3\zz}\rtimes_{\varphi} SL_{2}\Par{\zz/3\zz}$. The centralizer $C_{W\Par{E_{6}}}\Par{\Omega}$ and this subgroup both have order $648$, so they must be isomorphic. Since \[\mathcal{H}_{3}\Par{\zz/3\zz}\rtimes_{\varphi} SL_{2}\Par{\zz/3\zz}<\text{Im}\Par{\rho}<C_{W\Par{E_{6}}}\Par{\Omega}\]
the three groups must be isomorphic, concluding the proof. A public repository containing the Sage \cite{sagemath} and Mathematica \cite{Mathematica} computations employed can be found in the following link:
\begin{center} \href{https://github.com/amedranomdelc/Monodromy-of-the-Family-of-Cubic-Surfaces-Branching-over-Smooth-Cubic-Curves}{Computation Repository}
\end{center}

\begin{rmk}
	The problem studied in this paper lies within a much more general context, arising from the study of universal families of degree $d$ cyclic branched covers over smooth hypersurfaces of degree $n$ in $\pp^{N}$, where $d$ divides $n$.
	\begin{center}
	\begin{tikzcd}
	X_{d, n}\arrow[r, hookrightarrow] & \mathcal{E}_{d, n, N}\arrow[d] \\
	 & \mathcal{U}_{n, N}
	\end{tikzcd}
	\end{center}
	In this paper, we study the family corresponding to the case $(d, n, N)=(3, 3, 2)$. The problem of determining the monodromy of these families has been previously studied in lower dimension. In \cite{mcmullen}, McMullen provides a description of the family of degree $d$ cyclic branched covers branching over $n$ distinct points in $\pp^{1}$, given by a smooth subvariety of degree $n$ in $\pp^{1}$. This corresponds to the case $(d, n, N)=(d, n, 1)$. He then determines for which pairs $\Par{d, n}$ the induced monodromy map is surjective onto the corresponding automorphism group with the natural restrictions that come along with the branch cover structure.
\end{rmk}

\subsection{Acknowledgements}
I would like to thank Benson Farb for his invaluable guidance, advice and support through the making of this paper, as well as for suggesting this problem. I would also like to thank Olga Medrano Mart\'in del Campo for introducing me to Sage, tool which permitted many of the computations in this paper to be carried out. I'm grateful to Eduard Looijenga, Ronno Das, Nathaniel Mayer and Reid Harris for many helpful conversations, and to the Jump Trading Mathlab Fund for providing a working space where said conversations could be held. Finally, I would like to thank Curtis McMullen, Nick Salter, Dan Margalit, Maxime Bergeron and the referee for their helpful comments to make this paper more readable.

\section{Preliminaries}\label{Prelim}

\subsection{General facts about cubic surfaces}\label{FactsCubics}
We begin by recalling facts about smooth cubic surfaces which we wil employ through the paper. These facts can be found in e.g. \cite{hartshorne}. Let $X\subset \pp^{3}$ be a smooth cubic surface.

A classical result states all smooth cubic surfaces $X$ are blowups of $\pp^{2}$ at six points $p_{1}, \ldots, p_{6}$ in general position, and moreover $X$ contains $27$ lines. Therefore, the intersection form on $\cohomoldeg{2}{X}{\zz}$ is of type $\Par{1, 6}$ and we have the decomposition
	\[\cohomoldeg{2}{X}{\zz}\cong \cohomoldeg{2}{X}{\zz}_{+}\oplus \cohomoldeg{2}{X}{\zz}_{-}\cong \zz\oplus \zz^{6}.\]
The cohomology classes of any six pairwise disjoint lines in $X$ form a basis for $\cohomoldeg{2}{X}{\zz}_{-}$. The intersection pattern of these $27$ lines is given by the dual of the Schl\"afli graph, regarding each line as a vertex, and two lines intersect if and only if their corresponding vertices are joined by an edge. Let
\begin{center}
\begin{tikzcd}
Bl_{\WBPar{p_{1}, \ldots, p_{6}}}\Par{\pp^{2}}\arrow[d, "\pi"]\arrow[r, "\cong"] & X \\
\pp^{2}
\end{tikzcd}
\end{center}
be the blowup map from $X$ to $\pp^{2}$ at the points $p_{1}, \ldots, p_{6}$. Set $e_{0}, e_{1}, \ldots, e_{6}\in \cohomoldeg{2}{X}{\zz}$ as $e_{0}=\pi^{-1}\Par{\text{PD}\BPar{H}}$, the preimage of the Poincar\'e dual of the hyperplane class $\BPar{H}$ in $\pp^{2}$, and $e_{i}$ as the class of the exceptional divisor $L_{i}$ corresponding to $p_{i}$ for $i=1, \ldots, 6$. The divisors $L_{1}, \ldots, L_{6}$ constitute $6$ of the lines contained in $X$, and the remaining $21=15+6$ lines are given by blowups of
	\begin{itemize}
		\item the $15$ lines which pass through $p_{i}$ and $p_{j}$ for all pairs $i\neq j$, and
		\item the $6$ conics determined by $5$ of the blowup points $p_{i}$.
	\end{itemize}
The lines $L_{1}, \ldots, L_{6}$ are pairwise disjoint, and thus $e_{1}, \ldots, e_{6}$ form a basis for $\cohomoldeg{2}{X}{\zz}_{-}$. With respect to the basis $\WBPar{e_{0}, e_{1}, \ldots, e_{6}}$ of $\cohomoldeg{2}{X}{\zz}$, the intersection form is given by
	\begin{align*}
	\Par{\cdot, \cdot}_{X}:\cohomoldeg{2}{X}{\zz}\times \cohomoldeg{2}{X}{\zz}&\to \zz \\
	\Par{\mathbf{a}, \mathbf{b}}&\mapsto \mathbf{a}^{T}\begin{pmatrix}
	1 & 0 \\
	0 & -I_{6} \\
	\end{pmatrix}\mathbf{b}.
	\end{align*}
Let $L_{i, j}$ be the line in $X$ which is a blowup of the line passing through $p_{i}$ and $p_{j}$, and let $L_{i^{\ast}}$ be the line in $X$ which is a blowup of the conic passing through all blowup points except for $p_{i}$. Then
	\[
		\BPar{L_{i, j}}=e_{0}-e_{i}-e_{j}\quad\text{and}\quad\BPar{L_{i^{\ast}}}=2e_{0}+e_{i}-\Par{e_{1}+\cdots+e_{6}}.
	\]
In \cite{harris}, it is shown that the group of automorphisms of $\cohomoldeg{2}{X}{\zz}$ that preserve the intersection form is isomorphic to the Galois group of $X$, and these two groups are isomorphic to $W\Par{E_{6}}$. These automorphisms are determined by the images of the cohomology classes of any six disjoint lines, such as $e_{1}, \ldots, e_{6}$. Since the intersection form is preserved, this is equivalent to a permutation of the $27$ lines in $X$.

\subsection{Inflection points of $V\Par{f}$ versus lines in $V\Par{w^{3}-f}$}\label{CorrespondanceInflectionLines}
The lines on a cubic surface $X\cong V\Par{w^{3}-f}$ have a very particular structure, which can be described in terms of the inflection points of $f$. Since $f$ is a smooth cubic curve, it has $9$ inflection points, given by the intersection $V\Par{f}\cap V\Par{\det\text{Hess}\Par{f}}$, where
\[\text{Hess}\Par{f}=\begin{pmatrix}[1.5]
\frac{\partial^{2}f}{\partial x^{2}} & \frac{\partial^{2}f}{\partial x\partial y} &  \frac{\partial^{2}f}{\partial x\partial z}\\
\frac{\partial^{2}f}{\partial y\partial x} & \frac{\partial^{2}f}{\partial y^{2}} &  \frac{\partial^{2}f}{\partial y\partial z} \\
\frac{\partial^{2}f}{\partial z\partial x} &  \frac{\partial^{2}f}{\partial z\partial y} & \frac{\partial^{2}f}{\partial z^{2}} \\
\end{pmatrix}\]
is the Hessian matrix of $f$. The structure of the lines is described in the following proposition.

\begin{prop}\label{LineTriplesOverFlexes}
	Consider a smooth cubic curve $V\Par{f}\subset \pp^{2}$ and the branched cover
	\begin{center}
		\begin{tikzcd}
		V\Par{w^{3}-f} \arrow[d, "p", "\zz/3\zz"']  & \\
		\pp^{2}\arrow[r, hookleftarrow] & V\Par{f} 
		\end{tikzcd}
	\end{center}
	Let $P$ be an inflection point of $V\Par{f}$ and let $l_{P}\subset \pp^{2}$ be the tangent line to $V\Par{f}$ at $P$. Then $p^{-1}\Par{l_{P}}$ consists of $3$ concurrent lines at $p^{-1}\Par{P}$ in $V\Par{w^{3}-f}$. Namely, the $27$ lines conatained in $V\Par{w^{3}-f}$ lie over each of the $9$ inflection points of $f$ in concurrent triples.
\end{prop}

\begin{proof}
	Performing a change of coordinates, $f$ can be transformed into its \emph{Hesse normal form}, which is
\[f=x^{3}+y^{3}+z^{3}-3\mu xyz\]
for some $\mu$ such that $\mu^{3}\neq 1$. We have $\det \text{Hess}\Par{f}=216\Par{1-\mu^{3}}xyz$, so the inflection points of $f$ are
\begin{center}
	\begin{tabular}{c c c}
		$\BPar{1:-1:0}$ & $\BPar{1:-\omega:0}$ & $\BPar{1:-\omega^{2}:0}$ \\
		& & \\
		$\BPar{-1:0:1}$ & $\BPar{-\omega:0:1}$  & $\BPar{-\omega^{2}:0:1}$ \\
		& & \\
		$\BPar{0:1:-1}$ & $\BPar{0:1:-\omega}$  & $\BPar{0:1:-\omega^{2}}$
	\end{tabular}
\end{center}
The tangent line to $V\Par{f}$ at a point $P$ is given by the equation
\[\nabla f_{P}\cdot \Par{x, y, z}=x\frac{\partial f}{\partial x}\Par{P}+y\frac{\partial f}{\partial y}\Par{P}+z\frac{\partial f}{\partial z}\Par{P}=0\]
and for such an $f$ we have $\nabla f=\Par{3x^{2}-3\mu yz, 3y^{2}-3\mu zx, 3z^{2}-3\mu xy}$. Hence the tangent lines at the inflection points of $f$ are
\begin{center}
	\begin{tabular}{c c c}
		$V\Par{x+y+\mu z}$ & $V\Par{x+\omega^{2}y+\mu\omega z}$ & $V\Par{x+\omega y+\mu\omega^{2}z}$ \\
		$V\Par{x+\mu y+z}$ & $V\Par{\omega^{2}x+\mu \omega y+z}$ & $V\Par{\omega x+\mu \omega^{2}y+z}$ \\
		$V\Par{\mu x+y+z}$ & $V\Par{\mu \omega x+y+\omega^{2}z}$ & $V\Par{\mu \omega^{2}x+y+\omega z}$ 
	\end{tabular}
\end{center}
in correspondence with the inflection points shown above. Consider a tangent line $L$ at one of the inflection points $P$, say $L=V\Par{\mu x+y+z}$. We now proceed to determine its preimage $p^{-1}\Par{L}$ (the preimage for all remaining lines is determined in an analogous manner). Points in $p^{-1}\Par{L}$ satisfy $y=-\mu x-z$. Combining this equation along with $f$ we obtain
\[w^{3}-x^{3}-z^{3}+\Par{\mu x+z}^{3}+3\mu xz\Par{\mu x+z}=w^{3}-\Par{1-\mu^{3}}x^{3}=0\]
and thus, letting $\eta$ be a cube root of $1-\mu^{3}$, we have that $p^{-1}\Par{L}$ consists of three lines through $p^{-1}\Par{P}$ given by
\[V\Par{w-\omega^{n}\eta x}\cap V\Par{\mu x+y+z} \subset\pp^{3} \quad\text{ for }n=0, 1, 2.\]
This is analogous for the remaining tangent lines at the inflection points of $V\Par{f}$, so to each inflection point $P$ of $V\Par{f}$, we have associated $3$ lines in $V\Par{w^{3}-f}$ passing through $p^{-1}\Par{P}$.
\end{proof}

It should be emphasized that \Cref{LineTriplesOverFlexes} is the crucial property that characterises the cubic surfaces of the form $V(w^{3}-f)$, and this property will be exploited throughout this paper.

\section{Restricting $\text{Im}\Par{\rho}$}\label{RestrictingImageSection}
In this section, we show that the image of $\rho$ is contained in the centralizer of an order $3$ element $\Omega\in W\Par{E_{6}}$.

\subsection{The action $\Omega$}\label{ActionOfOmega}
Consider the action $\tau$ of $\mathcal{H}_{3}\Par{\zz/3\zz}$ on $V\Par{f}$ coming from the extension given in \cite{dolglib}, which acts on the inflection points of $V\Par{f}$. These inflection points correspond to the $3$-torsion points of $V\Par{f}$, or those points $P$ with $3P=0$, given an elliptic curve structure on it. Moreover, $\tau$ acts by translation on the $\zz/3\zz\times \zz/3\zz$ lattice formed by these $3$-torsion points. Namely, the translation on the $\zz/3\zz\times\zz/3\zz$ lattice is given by the $a$ and $b$ entries of a matrix element of $\mathcal{H}_{3}\Par{\zz/3\zz}$ as follows:
\begin{align*}
\tau:\mathcal{H}_{3}\Par{\zz/3\zz}\acts V\Par{f}\Par{\zz/3\zz}&\cong \zz/3\zz\times \zz/3\zz \\
\begin{pmatrix}
1 & a & c \\
0 & 1 & b\\
0 & 0 & 1
\end{pmatrix} \cdot \Par{x, y}&=\Par{x+a, y+b}.
\end{align*}
We will use this action in \Cref{ImageOfHeissenbergSubsection}, further explained from the Hessian normal form of a cubic curve. The action of an element $M\in \mathcal{H}_{3}\Par{\zz/3\zz}$ under $\tau$ comes from a  linear transformation in $\pp^{2}$ \cite{dolglib}. This action can be lifted to a linear transformation on $\pp^{3}$ giving an automorphism of $V\Par{w^{3}-f}$ which induces the element in $\text{Im}\Par{\rho}$ coming from $M$. Let $Z$ be a generator of the center $Z\Par{\mathcal{H}_{3}\Par{\zz/3\zz}}$. Lifting $Z$ to an action on $\pp^{3}$ gives an automorphism $T$ which fixes the inflection points in the curve $V\Par{f}\cap V\Par{w}$, and acts on $V\Par{w^{3}-f}$ by multiplication by $\omega^{-1}=e^{-2\pi i/3}$ on the $w$-coordinate:
\[T:[x:y:z:w]\mapsto [x:y:z:\omega^{-1} w]\quad \omega={e^{\frac{2\pi i}{3}}}.\]
This is precisely a generator of the $\zz/3\zz$ deck group action on $V\Par{w^{3}-f}$. Let $\Omega$ be the element of $\text{Im}\Par{\rho}$ induced by $T$. We can now prove the following proposition.

\begin{prop}\label{ImageInCentalizerOfOmega}
	With the notation above, $\text{Im}\Par{\rho}$ is contained in the centralizer $C_{\text{Aut}\Par{\cohomoldeg{2}{V\Par{w^{3}-f}}{\zz}}}\Par{\Omega}$.
\end{prop}

\begin{proof}
	The homomorphism $\Phi:SL_{2}\Par{\zz}\to \text{Aut}\Par{\mathcal{H}_{3}\Par{\zz/3\zz}}$ induced by the split group extension
\[1\to \mathcal{H}_{3}\Par{\zz/3\zz}\to \pi_{1}\Par{\mathcal{U}_{3}}\to SL_{2}\Par{\zz}\to 1\]
described in \cite{dolglib} is given by the composition
\begin{center}
	\begin{tikzcd}
	SL_{2}\Par{\zz}\arrow[r, "\mod{3}"]\arrow[rr, bend left=20, "\Phi"]& SL_{2}\Par{\zz/3\zz}\arrow[r, "\varphi"] & \text{Aut}\Par{\mathcal{H}_{3}\Par{\zz/3\zz}}
	\end{tikzcd}
\end{center}
where the map $\varphi$ is given by:
\begin{align*}
\varphi:SL_{2}\Par{\zz/3\zz}&\to \text{Aut}\Par{\mathcal{H}_{3}\Par{\zz/3\zz}} \\ 
M&\mapsto \BPar{\varphi_M:\begin{pmatrix}
1 & a & c \\
0 & 1 & b \\
0 & 0 & 1
\end{pmatrix}\mapsto
\begin{pmatrix}
1 & M_{1}\Par{a, b} & c \\
0 & 1 & M_{2}\Par{a, b} \\
0 & 0 & 1
\end{pmatrix}}\quad M\Par{a, b}=\Par{M_{1}\Par{a, b}, M_{2}\Par{a, b}}
\end{align*}
for all $M\in SL_{2}\Par{\zz/3\zz}$. Since the center of $\mathcal{H}_{3}\Par{\zz/3\zz}$ is given by
\[Z\Par{\mathcal{H}_{3}\Par{\zz/3\zz}}\cong \Bracket{\begin{pmatrix}
	1 & 0 & 1 \\
	0 & 1 & 0 \\
	0 & 0 & 1
	\end{pmatrix}}\]
it folows that $\Phi_{N}$ fixes $Z\Par{\mathcal{H}_{3}\Par{\zz/3\zz}}$ for every $N\in SL_{2}\Par{\zz}$. Therefore, $Z\Par{\mathcal{H}_{3}\Par{\zz/3\zz}}$ is in the center of $\pi_{1}\Par{\mathcal{U}_{3}}$. Since $\Omega\in \rho\Par{Z\Par{\mathcal{H}_{3}\Par{\zz/3\zz}}}$ by construction, it follows that $\Omega$ is in the center of $\text{Im}\Par{\rho}$, or equivalently, $\text{Im}\Par{\rho}$ is contained in the centralizer of $\Omega$.
\end{proof}

\subsection{Restriction to $W\Par{E_{6}}$}\label{InvarianceSubsection}

\begin{prop}\label{ImageInWE6}
	$\text{Im}\Par{\rho}$ is contained in $W\Par{E_{6}}$.
\end{prop}

\begin{proof}
The action of any element $\sigma\in \pi_{1}\Par{\mathcal{U}_{3}}$ on $V\Par{w^{3}-f}$ maps $V\Par{w^{3}-f}\cap V\Par{w}$ to itself, inducing a permutation on the set of $9$ inflection points of $V\Par{f}$. This implies that the preimages of the inflection points under the branched cover $p$ are permuted as well, and we have shown these are exactly the $27$ lines in $V\Par{w^{3}-f}$. Therefore, $\sigma$ maps lines to lines, and for any line $L\subset V\Par{w^{3}-f}$ we have
\[\rho\Par{\sigma}\Par{\BPar{L}}=\BPar{\sigma\Par{L}}.\]
The incidence of the $27$ lines contained in $V(w^{3}-f)$ is therefore preserved. The classes of these lines span $\cohomoldeg{2}{V\Par{w^{3}-f}}{\zz}$ and its intersection form is non-degenerate, so $\rho$ preserves said intersection form. In \cite{harris}, it is shown that automorphisms of $\cohomoldeg{2}{V\Par{w^{3}-f}}{\zz}$ preserving the form lie within the odd orthogonal group $O_{6}^{-}\Par{\zz/2\zz}$, which is isomorphic to the Weyl group $ W\Par{E_{6}}$. Hence, $\text{Im}\Par{\rho}\subset W\Par{E_{6}}$.
\end{proof}

Summarizing, \Cref{ImageInCentalizerOfOmega} and \Cref{ImageInWE6} give the following corollary.

\begin{cor}\label{ImageInCentralizer}
	We have $\text{Im}\Par{\rho}\subset C_{W\Par{E_{6}}}\Par{\Omega}.$
\end{cor}

\subsection{Computing $\Abs{C_{W\Par{E_{6}}}\Par{\Omega}}$}\label{ComputingCentralizerSubsection}

Consider the $W\Par{E_{6}}$-action on itself by conjugation. By the orbit-stabilizer theorem applied to $\Omega$, we have
\[51840=\Abs{W\Par{E_{6}}}=\Abs{\text{Orbit}_{W\Par{E_{6}}}\Par{\Omega}}\cdot \Abs{C_{W\Par{E_{6}}}\Par{\Omega}}\]
and the order of the orbit of $\Omega$ corresponds to the size of its conjugacy class. We will determine the size of this conjugacy class by looking at the character table of $W\Par{E_{6}}$, found for example in \cite{frame}. For this, we must compute the action of $\Omega$ on the set of lines of our chosen base point, the surface $V\Par{w^{3}-f}$ corresponding to the curve \[f=y^{2}z-x^{3}+xz^{2}.\]
With the help of \cite{sagemath}, we enumerate the $27$ lines in $V\Par{w^{3}-f}$. Then, we find a set of six pairwise disjoint lines $L_{1}, \ldots, L_{6}\subset V\Par{w^{3}-f}$, and proceed to compute the intersection pattern of the remaining $21$ lines with each $L_{i}$ in order to compute their cohomology classes. We obtain that for our choice of lines, $\Omega$ acts by
\begin{align*}
	L_{1}&\mapsto L_{5^{\ast}} &
	L_{2}&\mapsto L_{2, 3} &
	L_{3}&\mapsto L_{3, 6} \\
	L_{4}&\mapsto L_{1^{\ast}}  &
	L_{5}&\mapsto L_{4^{\ast}}  &
	L_{6}&\mapsto L_{2, 6}
\end{align*}
following the notation in \Cref{FactsCubics}. Therefore, with respect to the basis $\WBPar{e_{0}, \BPar{L_{1}}, \ldots, \BPar{L_{6}}}$ we have
\[\Omega=\begin{pmatrix*}[r]
4 & 2 & 1 & 1 & 2 & 2 & 1 \\
-1 & -1 & 0 & 0 & 0 & -1 & 0 \\
-2 & -1 & -1 & 0 & -1 & -1 & -1 \\
-2 & -1 & -1 & -1 & -1 & -1 & 0 \\
-1 & -1 & 0 & 0 & -1 & 0 & 0 \\
-1 & 0 & 0 & 0 & -1 & -1 & 0 \\
-2 & -1 & 0 & -1 & -1 & -1 & -1
\end{pmatrix*}\in \text{Aut}\Par{\cohomoldeg{2}{V\Par{w^{3}-f}}{\zz}}.\]
Now consider the complex $W\Par{E_{6}}$-representation given by
\[\cc^{7}\cong \cohomoldeg{2}{V\Par{w^{3}-f}}{\cc}\cong \cohomoldeg{2}{V\Par{w^{3}-f}}{\zz}\otimes \cc\]
This representation contains a copy of the trivial representation, since the canonical class is fixed by $W\Par{E_{6}}$, and a copy of an irreducible $6$-dimensional representation. Namely, 
\[\cohomoldeg{2}{V\Par{w^{3}-f}}{\cc}\cong \cc_{\text{triv}}\oplus V_{6}.\]
Since $\text{Trace}\Par{\Omega}=-2$, at the level of characters we have
\[-2=\chi_{\cohomoldeg{2}{V\Par{w^{3}-f}}{\cc}}\Par{\Omega}=\chi_{\cc_{\text{triv}}}\Par{\Omega}+\chi_{V_{6}}\Par{\Omega}=1+\chi_{V_{6}}\Par{\Omega}\]
and therefore $\chi_{V_{6}}\Par{\Omega}=-3$. There exist two $6$-dimensional irreducible representations of $W\Par{E_{6}}$, both of which have a unique conjugacy class whose character equals $-3$. Therefore, $\Omega$ must belong to this conjugacy class, which has order $80$. Thus, we conclude
\[\Abs{C_{W\Par{E_{6}}}\Par{\Omega}}=\frac{\Abs{W\Par{E_{6}}}}{\Abs{\text{Orbit}_{W\Par{E_{6}}}\Par{\Omega}}}=\frac{51840}{80}=648.\]

\section{Explicit computations for $f=y^{2}z-x^{3}+xz^{2}$}\label{ExplicitComputationsSection}

We have constructed via \Cref{LineTriplesOverFlexes} a correspondence between the inflection points of a smooth cubic curve $V\Par{f}$ and triples of lines in its associated surface $V\Par{w^{3}-f}$ lying over the inflection points of $V\Par{f}$. The goal of this section is to compute this datum explicitly for a choice of $f$ serving as a base point in $\mathcal{U}_{3}$. This datum is then manipulated with the aid of computer software (\cite{sagemath} and \cite{Mathematica}), and used in \Cref{ComputingImagesSection} to explicitly compute generators of $\text{Im}\Par{\rho}$.

\subsection{Inflection points of $f$}\label{InflectionPointsSubsection}
Consider the curve $f=y^{2}z-x^{3}+xz^{2}$. Then
\[\nabla f=\begin{pmatrix}
	-3x^{2}+z^{2} \\
	2yz \\
	y^{2}+2xz
\end{pmatrix}\quad \text{and}\quad \text{Hess}\Par{f}=\begin{pmatrix}
	-6x & 0 & 2z \\
	0 & 2z & 2y \\
	2z & 2y & 2x
\end{pmatrix}.\]
Hence, $\det \text{Hess}\Par{f}=8\Par{3x\Par{y^{2}-xz}-z^{3}}$ and the inflection points of $f$ are given by
\[V\Par{y^{2}z-x^{3}+xz^{2}}\cap V\Par{3xy^{2}-3x^{2}z-z^{3}}.\]
Let $\BPar{x:y:z}$ be an inflection point of $f$.
\begin{itemize}
	\item If $z=0$, then $f=-x^{3}=0$ so $x=0$, and thus $\BPar{x:y:z}=\BPar{0:1:0}$.
	\item If $z\neq 0$, then one may assume $z=1$. Then $y^{2}=x^{3}-x$ and $3xy^{2}-3x^{2}-1=0$. Substituting $y^{2}$ in the second equation we obtain
	\[R\Par{x}=3x^{4}-6x^{2}-1=0.\]
	Hence, the remaining $8$ inflection points are $\BPar{\alpha:\pm \sqrt{\alpha^{3}-\alpha}:1}$ with $\alpha$ a root of $R\Par{x}$. Namely,
	\[\alpha\in \WBPar{\pm a, \pm ia\frac{\Par{\sqrt{3}-1}^{2}}{2}}\quad\text{where }a=\sqrt{\frac{3+2\sqrt{3}}{3}}.\]
	To describe the $y$-coordinate of these inflection points, note that $\alpha^{2}-1=\pm\frac{2\sqrt{3}}{3}$, so
	\[y\in \WBPar{\pm b, \pm ib, \pm b\Par{1+i}\Par{\sqrt{3}-1}, \pm b\Par{1-i}\Par{\sqrt{3}-1}}\quad\text{where }b=\sqrt{a\frac{2\sqrt{3}}{3}}.\]
\end{itemize}
Therefore, the inflection points of $f=y^{2}z-x^{3}+xz^{2}$ are
\begin{center}
	\begin{tabular}{c c c}
		$\BPar{-a:+ib:1}$ & $\BPar{-ia\frac{\Par{\sqrt{3}-1}^{2}}{2}:-b\Par{1+i}\Par{\sqrt{3}-1}:1}$ & $\BPar{+ia\frac{\Par{\sqrt{3}-1}^{2}}{2}:-b\Par{1-i}\Par{\sqrt{3}-1}:1}$ \\
		&&\\
		$\BPar{-a:-ib:1}$ & $\BPar{+ia\frac{\Par{\sqrt{3}-1}^{2}}{2}:+b\Par{1-i}\Par{\sqrt{3}-1}:1}$ & $\BPar{-ia\frac{\Par{\sqrt{3}-1}^{2}}{2}:+b\Par{1+i}\Par{\sqrt{3}-1}:1}$ \\
		&&\\
		$\BPar{0:1:0}$ & $\BPar{+a:+b:1}$ & $\BPar{+a:-b:1}$ \\
		&&\\
	\end{tabular}
\end{center}
and we present them in this way since these inflection points are in direct correspondence with the inflection points of the Hesse normal form $f_{H}$ of $f=y^{2}z-x^{3}+xz^{2}$ as shown in the proof of \Cref{LineTriplesOverFlexes} after transforming $V(f)$ to $V(f_{H})$ via the linear map $A:\pp^{2}\to \pp^{2}$ given by
\[A=\begin{pmatrix}
	1 & 0 & -a \\
	-\frac{\sqrt{3}+1}{2} & -i\sqrt[4]{\frac{3+2\sqrt{3}}{4}} & -a\Par{\frac{\sqrt{3}-1}{2}} \\
	-\frac{\sqrt{3}+1}{2} & i\sqrt[4]{\frac{3+2\sqrt{3}}{4}} & -a\Par{\frac{\sqrt{3}-1}{2}} \\	
\end{pmatrix}.\]
We use this transformation to compute the images of the generators of $\mathcal{H}_{3}\Par{\zz/3\zz}<\pi_{1}\Par{\mathcal{U}_{3}}$, as \cite{dolglib} decribe the action of $\mathcal{H}_{3}\Par{\zz/3\zz}$ in terms on the Hessian form of cubic curves.

\subsection{Lines in $V\Par{w^{3}-f_{\lambda}}$}\label{LinesSubsection}

Consider a family of variations of $f$ parametrized by $\lambda\in \cc\setminus \WBPar{\pm1}$, given by
\[
f_{\lambda}:=y^{2}z-\Par{x-z}\Par{x+z}\Par{x-\lambda z}=y^{2}z-x^{3}+\lambda x^{2}z+xz^{2}-\lambda z^{3}
\]
where $f=f_{0}$. We now proceed to compute the inflection points of $V\Par{f_{\lambda}}$ and the lines in $V\Par{w^{3}-f_{\lambda}}$, following \Cref{InflectionPointsSubsection} and the proof of \Cref{LineTriplesOverFlexes}, respectively. These calculations are crucial for the computation of $\rho\Par{SL_{2}\Par{\zz}}$, as we shal see in \Cref{ImageOfSL2Subsection}. To compute the inflections points of $f_{\lambda}$, we use
\[\nabla f_{\lambda}=\begin{pmatrix}
-3x^{2}+z^{2}+2\lambda xz \\
2yz \\
y^{2}+2xz-3\lambda z^{2}
\end{pmatrix}\quad \text{and}\quad \text{Hess}\Par{f_{\lambda}}=\begin{pmatrix}
-6x+2\lambda z & 0 & 2z+2\lambda x \\
0 & 2z & 2y \\
2z+2\lambda x & 2y & 2x-6\lambda z
\end{pmatrix}.\]
Hence, $\frac{1}{8}\det \text{Hess}\Par{f}=\Par{3x-\lambda z}\Par{y^{2}-xz+3\lambda z^{2}}-z\Par{z+\lambda x}^{2}$ so the inflection points of $f$ are given by
\[V\Par{y^{2}z-x^{3}+\lambda x^{2}z+xz^{2}-\lambda z^{3}}\cap V\Par{\Par{3x-\lambda z}y^{2}-z\Par{x^{2}\Par{3+\lambda^{2}}-8\lambda xz+z^{2}\Par{1+3\lambda^{2}}}}.\]
Let $\BPar{x:y:z}$ be an inflection point.
\begin{itemize}
	\item If $z=0$, then $f=-x^{3}=0$ so $x=0$, and thus $\BPar{x:y:z}=\BPar{0:1:0}$.
	\item If $z\neq 0$, then one may assume $z=1$. Then
	\begin{align*}
	y^{2}&=x^{3}-\lambda x^{2}-x+\lambda \\
	0&=\Par{3x-\lambda}y^{2}-\Par{3+\lambda^{2}}x^{2}+8\lambda x-1-3\lambda^{2}
	\end{align*}
	Substituting $y^{2}$ in the second equation we obtain
	\[R_{\lambda}\Par{x}=3x^{4}-4\lambda x^{3}-6x^{2}+12\lambda x-1-4\lambda^{2}=0.\]
	Hence, the remaining $8$ inflection points are $\BPar{\alpha:\pm \sqrt{\alpha^{3}-\lambda \alpha^{2}-\alpha+\lambda}:1}$ with $\alpha$ a root of $R_{\lambda}\Par{x}$.
\end{itemize}
We may now proceed to compute the triples of lines in $V\Par{w^{3}-f_{\lambda}}$ above each inflection point in $V\Par{f_{\lambda}}$.

\begin{itemize}
	\item At the inflection point $\BPar{0:1:0}$, we have the tangent line is given by $z=0$. This gives \[w^{3}-f_{\lambda}=w^{3}+x^{3}=0\]
	so the three lines lying above $\BPar{0:1:0}$ are
	\[V\Par{w+\omega^{n}x}\cap V\Par{z}\quad \text{ for }n=0, 1, 2.\]
	
	\item At an inflection point of the form $P=\BPar{\alpha:\sqrt{\alpha^{3}-\lambda \alpha^{2}-\alpha+\lambda}:1}$, the tangent line at $P$ is given by
	\[\Par{-3\alpha^{2}+2\lambda \alpha+1}x+\Par{2\sqrt{\alpha^{3}-\lambda \alpha^{2}-\alpha+\lambda}}y+\Par{\alpha^{3}-\lambda\alpha^{2}+\alpha-2\lambda}z=0.\]
	To compute the three lines lying over $P$ in $V\Par{w^{3}-f_{\lambda}}$, we substitute
	\[y^{2}=\frac{\Par{\Par{-3\alpha^{2}+2\lambda \alpha+1}x+\Par{\alpha^{3}-\lambda\alpha^{2}+\alpha-2\lambda}z}^{2}}{4\Par{\alpha^{3}-\lambda \alpha^{2}-\alpha+\lambda}}\]
	in $f_{\lambda}$, and using that $R_{\lambda}\Par{\alpha}=0$ to simplify, we obtain
	\[w^{3}-f_{\lambda}=w^{3}+\Par{x-\alpha z}^{3}=0.\]
	Hence, the lines lying over $P$ are
	\[
	V\Par{w+\omega^{n}\Par{x-\alpha z}}\cap V\Par{\Par{-3\alpha^{2}+2\lambda \alpha+1}x+\Par{2\sqrt{\alpha^{3}-\lambda \alpha^{2}-\alpha+\lambda}}y+\Par{\alpha^{3}-\lambda\alpha^{2}+\alpha-2\lambda}z}
	\]
	for $n=0, 1, 2$.
	
\end{itemize}

\section{Computing generators of $\text{Im}\Par{\rho}$}\label{ComputingImagesSection}

The goal of this section is to use the datum computed in \Cref{ExplicitComputationsSection} to determine explicitly the images under $\rho$ of $4$ elements of $\pi_{1}\Par{\mathcal{U}_{3}}$. The images computed will serve a posteriori as generators of $\text{Im}\Par{\rho}$. The $4$ elements for which we choose to compute their images come from the semidirect product structure of $\pi_{1}\Par{\mathcal{U}_{3}}$, 
\[\pi_{1}\Par{\mathcal{U}_{3}}\cong \mathcal{H}_{3}\Par{\zz/3\zz}\rtimes_{\varphi}SL_{2}\Par{\zz}\] and will, again a posteriori, provide the semidirect product structure of $\text{Im}\Par{\rho}$. 

\subsection{Images of generators of $\mathcal{H}_{3}\Par{\zz/3\zz}$}\label{ImageOfHeissenbergSubsection}
The transformation $A:\pp^{2}\to \pp^{2}$ introduced in \Cref{InflectionPointsSubsection} maps the curve $f=y^{2}z-x^{3}+xz^{2}$ to its Hesse normal form, $f_{H}=x^{3}+y^{3}+z^{3}-3\mu xyz$, where $\mu=\sqrt{3}+1$. The map $A$ induces a map
\[A'=\begin{pmatrix}
A & 0 \\
0 & -\eta
\end{pmatrix}:\pp^{3}\to \pp^{3}\]
where $\eta$ is a cube root of $1-\mu^{3}$. Since $\mu$ is real, we can take $\eta$ to be real for convenience. The map $A'$ maps the surface $V\Par{w^{3}-f}$ to $V\Par{w^{3}-f_{H}}$, and being a linear map, it defines a mapping between the $27$ lines of one surface to the other. Namely, we have
\begin{center}
	\begin{tikzcd}
	V\Par{w^{3}-f}\arrow[rr, "A'"] & & V\Par{w^{3}-f_{H}} \\
	L\arrow[u, hookrightarrow]\arrow[rr, "A'"] & & A'\Par{L}=L_{H}\arrow[u, hookrightarrow]
	\end{tikzcd}
\end{center}
and with the help of \cite{sagemath}, we compute the mapping of the lines as depicted in \Cref{HesseFormLinesMap}. For each surface, a set of $6$ pairwise non-intersecting lines $\{L_{1}, \ldots, L_{6}\}$ is found along with their incidences with the remaining $21$ lines. This allows us to compute the classes of the $27$ lines with respect to the basis $e_{0}, \BPar{L_{1}}, \ldots, \BPar{L_{6}}$ of $\cohomoldeg{2}{V\Par{w^{3}-f}}{\zz}$ and $\cohomoldeg{2}{V\Par{w^{3}-f_{H}}}{\zz}$ respectively.

\begin{figure}[!h]
	\centering
	\begin{tabular}{lll|lll|lll}
		&& & && & && \\
		&$L_{3}$& & &$L_{5, 6}$& & &$L_{4, 5}$& \\
		&$L_{3, 6}$& & &$L_{3, 4}$& & &$L_{5}$& \\
		&$L_{6^{\ast}}$& & &$L_{1, 2}$& & &$L_{4^{\ast}}$& \\
		&& & && & && \\
		\hline
		&& & && & && \\
		&$L_{1^{\ast}}$& & &$L_{2, 6}$& & &$L_{2, 4}$& \\
		&$L_{1, 4}$& & &$L_{2^{\ast}}$& & &$L_{1, 6}$& \\
		&$L_{4}$& & &$L_{6}$& & &$L_{3, 5}$& \\
		&& & && & && \\
		\hline
		&& & && & && \\
		&$L_{1, 3}$& & &$L_{1 }$& & &$L_{3^{\ast}}$& \\
		&$L_{2, 5}$& & &$L_{5^{\ast}}$& & &$L_{2}$& \\
		&$L_{4, 6}$& & &$L_{1, 5}$& & &$L_{2, 3}$& \\
		&& & && & &&
	\end{tabular}
	$\quad\xrightarrow{A'}\quad$
	\begin{tabular}{lll|lll|lll}
		&& & && & && \\
		&$L_{3}$& & &$L_{3, 4}$& & &$L_{4^{\ast}}$& \\
		&$L_{3, 6}$& & &$L_{1, 2}$& & &$L_{4, 5}$& \\
		&$L_{6^{\ast}}$& & &$L_{5, 6}$& & &$L_{5}$ \\
		&& & && & && \\
		\hline
		&& & && & && \\
		&$L_{1^{\ast}}$& & &$L_{6}$& & &$L_{1, 6}$& \\
		&$L_{1, 4}$& & &$L_{2, 6}$& & &$L_{3, 5}$& \\
		&$L_{4}$& & &$L_{2^{\ast}}$& & &$L_{2, 4}$& \\
		&& & && & && \\
		\hline
		&& & && & && \\
		&$L_{1, 3}$& & &$L_{1 }$& & &$L_{3^{\ast}}$& \\
		&$L_{2, 5}$& & &$L_{5^{\ast}}$& & &$L_{2}$& \\
		&$L_{4, 6}$& & &$L_{1, 5}$& & &$L_{2, 3}$& \\
		&& & && & &&
	\end{tabular}
	\caption{Mapping of the lines in $V\Par{w^{3}-f}$ (left) to lines in $V\Par{w^{3}-f_{H}}$ (right) induced by $A'$. The lines are determined by their \emph{position} in the diagram, and each labeled line on the left diagram is mapped to the line with the same label on the right diagram. Each box is in correspondance to the inflection points shown in \Cref{ExplicitComputationsSection}, and lines in each box lie over the same inflection point of $f$ (left) and $f_{H}$ (right). Moreover, each line has a defining equation which depends on a power of $\omega=e^{2\pi i/3}$. Within each box, the powers $\omega^{n}$ in the defining equations of the lines are $1, \omega, \omega^{2}$ from top to bottom.}
	\label{HesseFormLinesMap}
\end{figure}

The purpose of $A'$ is to determine the $\mathcal{H}_{3}\Par{\zz/3\zz}$-action on $V\Par{f}$ from the corresponding $\mathcal{H}_{3}\Par{\zz/3\zz}$-action on $V\Par{f_{H}}$. The later action is explicitly described in \cite{dolglib}. It is given by translation on the lattice of inflection points of $f_{H}$, and it is generated by
\[X=\begin{pmatrix}
1 & 0 & 0 \\
0 & \omega & 0 \\
0 & 0 & \omega^{2} \\
\end{pmatrix}\quad\text{and}\quad
Y=\begin{pmatrix}
0 & 1 & 0 \\
0 & 0 & 1 \\
1 & 0 & 0 \\
\end{pmatrix}
\quad\text{in }PSL_{3}\Par{\cc}.\]
These can be lifted to automorphisms of $V\Par{w^{3}-f_{H}}$ as
\[X'=\begin{pmatrix}
1 & 0 & 0 & 0 \\
0 & \omega & 0 & 0 \\
0 & 0 & \omega^{2} & 0 \\
0 & 0 & 0 & 1
\end{pmatrix}\quad\text{and}\quad
Y'=\begin{pmatrix}
0 & 1 & 0 & 0 \\
0 & 0 & 1 & 0 \\
1 & 0 & 0 & 0 \\
0 & 0 & 0 & 1
\end{pmatrix}
\quad\text{in }PSL_{4}\Par{\cc}\]
and these last two maps provide a permutation on the lines of $V\Par{w^{3}-f_{H}}$. This can be translated into a permutation of the lines in $V\Par{w^{3}-f}$ with the map $A'$. Thus, with the help of \Cref{HesseFormLinesMap} we obtain:

\begin{enumerate}
	\item The map $X'$ induces a permutation $H_{1}\in W\Par{E_{6}}$ which maps
	\begin{align*}
	L_{1}&\mapsto L_{6} &
	L_{2}&\mapsto L_{3, 5} & 
	L_{3}&\mapsto L_{1, 3} \\
	L_{4}&\mapsto L_{6^{\ast}} &
	L_{5}&\mapsto L_{2, 3} &
	L_{6}&\mapsto L_{3, 4}
	\end{align*}
	
	\item The map $Y'$ induces a permutation $H_{2}\in W\Par{E_{6}}$ which maps
	\begin{align*}
	L_{1}&\mapsto L_{3^{\ast}} &
	L_{2}&\mapsto L_{2, 5} &
	L_{3}&\mapsto L_{1, 2} \\
	L_{4}&\mapsto L_{2, 6} &
	L_{5}&\mapsto L_{3} &
	L_{6}&\mapsto L_{2, 4}
	\end{align*}
\end{enumerate}
and therefore, with respect to our basis $\WBPar{e_{0}, \BPar{L_{1}}, \ldots, \BPar{L_{6}}}$ we obtain the matrices
\[H_{1}=\begin{pmatrix*}[r]
3 & 0 & 1 & 1 & 2 & 1 & 1 \\
-1 & 0 & 0 & -1 & -1 & 0 & 0 \\
-1 & 0 & 0 & 0 & -1 & -1 & 0 \\
-2 & 0 & -1 & -1 & -1 & -1 & -1 \\
-1 & 0 & 0 & 0 & -1 & 0 & -1 \\
-1 & 0 & -1 & 0 & -1 & 0 & 0 \\
0 & 1 & 0 & 0 & 0 & 0 & 0
\end{pmatrix*}
\quad\text{and}\quad
H_{2}=\begin{pmatrix*}[r]
3 & 2 & 1 & 1 & 1 & 0 & 1 \\
-1 & -1 & 0 & -1 & 0 & 0 & 0 \\
-2 & -1 & -1 & -1 & -1 & 0 & -1 \\
0 & 0 & 0 & 0 & 0 & 1 & 0 \\
-1 & -1 & 0 & 0 & 0 & 0 & -1 \\
-1 & -1 & -1 & 0 & 0 & 0 & 0 \\
-1 & -1 & 0 & 0 & -1 & 0 & 0
\end{pmatrix*}\]
where
\[\mathbf{H}:=\Bracket{H_{1}, H_{2}}=\rho\Par{\mathcal{H}_{3}\Par{\zz/3\zz}}<W\Par{E_{6}}.\]
With the help of \cite{Mathematica}, we compute the size of the matrix group generated by $H_{1}$ and $H_{2}$, and we obtain that $\Abs{\mathbf{H}}=\Abs{\mathcal{H}_{3}\Par{\zz/3\zz}}=27$. Therefore, $\mathbf{H}\cong \mathcal{H}_{3}\Par{\zz/3\zz}$.

\subsection{Images of elements coming from $SL_{2}\Par{\zz}$}\label{ImageOfSL2Subsection}
Now we use our results from \Cref{LinesSubsection} to compute the image of two elements in $SL_{2}\Par{\zz}<\pi_{1}\Par{\mathcal{U}_{3}}$. Recall that by Hurwitz's theorem, every smooth cubic curve  is realized as a double branched cover of $\pp^{1}$, with branch locus equal to $4$ distinct points $\WBPar{a, b, c, d}\in\pp^{1}$. With a suitable change of coordinates given by a projective linear transformation in $\pp^{2}$, the branch points $\WBPar{a, b, c}$ may be mapped to $\WBPar{\pm 1, \infty}$, and the remaining branch point $d$ is mapped to some $\lambda\in \cc\setminus\WBPar{\pm 1}$. Therefore, a cubic curve can be embedded in $\pp^{2}$ as a curve $f_{\lambda}$ in the family defined in \Cref{LinesSubsection}
\[f_{\lambda}=y^{2}z-\Par{x-z}\Par{x+z}\Par{x-\lambda z}\quad \text{ with }\lambda\in \cc\setminus \WBPar{\pm 1}\]
which is the projectivization of the elliptic curve $y^{2}=\Par{x-1}\Par{x+1}\Par{x-\lambda}$. The curve $f=y^{2}z-x^{3}+xz^{2}$ is in such form letting $\lambda=0$, and the space $\cc\setminus\WBPar{\pm 1}$ parametrizes a subset of curves in $\mathcal{U}_{3}$ containing $f$. A loop based at $0$ in $\cc\setminus\WBPar{\pm 1}$ naturally parametrizes a loop based at $f$ in $\mathcal{U}_{3}$ via the map $\lambda\mapsto f_{\lambda}$. Since every cubic curve $f_{\lambda}$ has an inflection point at $\BPar{0:1:0}$, this inflection point is fixed by any automorphism of $f$ coming from elements in $\pi_{1}\Par{\mathcal{U}_{3}}$ given by loops in $\mathcal{U}_{3}$ consisting of curves of the form $f_{\lambda}$. This implies that for such loops, the action of $\mathcal{H}_{3}\Par{\zz/3\zz}<\pi_{1}\Par{\mathcal{U}_{3}}$ on $\cohomoldeg{2}{V\Par{w^{3}-f}}{\zz}$ is trivial, as there is no translation on the inflection points of $f$. Therefore these loops come from $SL_{2}\Par{\zz}<\pi_{1}\Par{\mathcal{U}_{3}}$. In light of this, we consider the generators of $\pi_{1}\Par{\cc\setminus\WBPar{\pm 1}}$ depicted in \Cref{TwoLoops} along with the corresponding loops in $\mathcal{U}_{3}$ parametrized by these generators. Our goal is now to compute the image under $\rho$ of these loops in $\pi_{1}\Par{\mathcal{U}_{3}}$.

\begin{figure}[!h]
	\centering 
	\begin{tikzpicture}
		\draw[->] (-3, 0) -- (3, 0) coordinate (x axis);
		\draw[->] (0, -1.5) -- (0, 1.5) coordinate (y axis);
		\draw[red, dashed, very thick, ->] (-1, 1) arc (-270:90:1cm);
		\draw[blue, dashed, very thick, ->] (1, -1) arc (-90:270:1cm);
		\filldraw (0, 0) circle (2pt) node[right]{0};
		\draw (-1, 0) node{$\times$};
		\draw (-1, 0) node[above]{$-1$};
		\draw (1, 0) node{$\times$};
		\draw (1, 0) node[above]{$1$};
		\draw (-2, -1) node{$\gamma_{-}$};
		\draw (2, -1) node{$\gamma_{+}$};
		\draw (4, .5) node[right]{$\gamma_{-}:\BPar{0, 1}\to \cc\setminus\WBPar{\pm 1}\quad t\mapsto -1+e^{2\pi it}$};
		\draw (4, -.5) node[right]{$\gamma_{+}:\BPar{0, 1}\to \cc\setminus\WBPar{\pm 1}\quad t\mapsto 1-e^{2\pi it}$};
	\end{tikzpicture}
	\caption{Generators $\gamma_{-}$, $\gamma_{+}$ of $\pi_{1}\Par{\cc\setminus\WBPar{\pm 1}, 0}$.}
	\label{TwoLoops}
\end{figure}

A loop $\gamma\in \pi_{1}\Par{\cc\setminus \WBPar{\pm 1}}$ induces a permutation on the $27$ lines of $V\Par{w^{3}-f}$ as follows: in \Cref{ExplicitComputationsSection} we prove that the $24$ of the lines in $V\Par{w^{3}-f_{\lambda}}$ are given by
\[
V\Par{w+\omega^{n}\Par{x-\alpha z}}\cap V\Par{\Par{-3\alpha^{2}+2\lambda \alpha+1}x+\Par{2\sqrt{\alpha^{3}-\lambda \alpha^{2}-\alpha+\lambda}}y+\Par{\alpha^{3}-\lambda\alpha^{2}+\alpha-2\lambda}z}
\]
and the remaining $3$ are given by $V\Par{w+\omega^{n}x}\cap V\Par{z}$. Letting $\lambda=\gamma\Par{t}$, $\alpha$ and $\lambda$ vary continuously in $t$. Indeed, $\alpha$ is a root of $R_{\lambda}\Par{x}$, a quartic polynomial whose coefficients are polynomials in $\lambda$. This implies that along the surfaces $V\Par{w^{3}-f_{\lambda}}$, these lines vary continuously. For a line $L\subset V\Par{w^{3}-f}$, let $L\Par{\gamma, t}$ be the line in $V\Par{w^{3}-f_{\gamma\Par{t}}}$ obtained by varying $L$ along $\gamma$. Thus, $\gamma$ induces a permutation given by
\[L\Par{\gamma, 0}\mapsto L\Par{\gamma, 1}.\]
To compute these permutations explicitly, we recall that lines in $V\Par{w^{3}-f_{\lambda}}$ lie over the inflection points of $f_{\lambda}$ in triples. The lines $V\Par{w+\omega^{n}x}\cap V\Par{z}$ are fixed along any path $\gamma$, as the coefficients of their defining equations are constant functions of $\lambda$. Thus, it remains to study what happens to the remaining $8$ triples of lines. Let $p\Par{\gamma, t}$ be the inflection point of $f_{\gamma\Par{t}}$ over which the line $L\Par{\gamma, t}$ lies. Then, we have the following proposition.

\begin{prop}\label{FlexPointsMoveLines}
	The permutation $L\Par{\gamma, 0}\mapsto L\Par{\gamma, 1}$ induces a permutation $p\Par{\gamma, 0}\mapsto p\Par{\gamma, 1}$ on the inflection points of $f$. Moreover, the permutation $L\Par{\gamma, 0}\mapsto L\Par{\gamma, 1}$ is completely determined by its induced permutation $p\Par{\gamma, 0}\mapsto p\Par{\gamma, 1}$.
\end{prop}

\begin{proof}
	Since any loop $\gamma$ induces an automorphism of $V\Par{f_{\lambda}}$, the permutation $L\Par{\gamma, 0}\mapsto L\Par{\gamma, 1}$ induces a permutation $p\Par{\gamma, 0}\mapsto p\Par{\gamma, 1}$ by restriction to the inflection point of $f$ contained in each line. This implies that the lines over each inflection point $p\Par{\gamma, 0}$ in $f$ are mapped to the lines over $p\Par{\gamma, 1}$.
	
	Moreover, in the defining equation $w+\omega^{n}\Par{x-\alpha z}$ of any of these $24$ lines, the $x$-coefficient $\omega^{n}$ is a constant function of $\lambda$.  Therefore this coefficient remains constant in the defining equations of $L\Par{\gamma, 0}$ and $L\Par{\gamma, 1}$ for any line $L$ and any loop $\gamma$. The permutation on the lines is simply determined by the defining equation
	\[V\Par{\Par{-3\alpha^{2}+2\lambda \alpha+1}x\pm\Par{2\sqrt{\alpha^{3}-\lambda \alpha^{2}-\alpha+\lambda}}y+\Par{\alpha^{3}-\lambda\alpha^{2}+\alpha-2\lambda}z}.\]
	In particular, the $y$-coordinate of each of the $8$ inflection points of $f$ distinct from $\BPar{0:1:0}$ distinguishes each inflection point. It is also a scalar multiple of the $y$-coefficient in the defining equations of the $24$ lines over the inflection points distinct from $\BPar{0:1:0}$, and this coefficient distinguishes such defining equations. Therefore the defining equation of our lines changes according to this coefficient, which changes according to the inflection points of $f$.
	This shows that the induced permutation $L\Par{\gamma, 0}\mapsto L\Par{\gamma, 1}$ is completely determined by its associated permutation on the inflection points $p\Par{\gamma, 0}\mapsto p\Par{\gamma, 1}$.
\end{proof}

Hence it suffices to study the permutation $p\Par{\gamma, 0}\mapsto p\Par{\gamma, 1}$. To do so, we simply study the permutation of the $y$-coordinates of these inflection points. We compute this with the help of \cite{Mathematica} as follows:

Given a path $\gamma$, $y$-coordinate of the inflection points of $f_{\gamma\Par{t}}$ distinct from $\BPar{0:1:0}$ are determined by the roots of the polynomial $R_{\gamma\Par{t}}$, which are the $x$-coordinates of these inflection points. We plot
\begin{align*}
\mathcal{X}_{N}\Par{\gamma}&:=\WBPar{\Par{x, \frac{k}{N}}\in \cc\times \BPar{0, 1}\mid k=0, 1, \ldots, N\text{ and }\BPar{x:y:1}\text{ is an inflection point of }f_{\gamma\Par{\frac{k}{N}}}} \\
\mathcal{Y}_{N}\Par{\gamma}&:=\WBPar{\Par{y, \frac{k}{N}}\in \cc\times \BPar{0, 1}\mid k=0, 1, \ldots, N\text{ and }\BPar{x:y:1}\text{ is an inflection point of }f_{\gamma\Par{\frac{k}{N}}}}
\end{align*}
Uniform continuity of the map $t\mapsto p\Par{\gamma, t}$ guarantees that for a sufficiently large integer $N$, the distance between the $y$-coordinates corresponding to the inflection points $p\Par{\gamma, \frac{k}{N}}$ and $p\Par{\gamma, \frac{k+1}{N}}$ is bounded by a fixed $\varepsilon>0$ for every $k$. Therefore the $y$-coordinates of the points $p\Par{\gamma, \frac{k}{N}}$ can be determined from the starting point $p\Par{\gamma, 0}$. This is illustrated for the curves $\gamma_{-}$ and $\gamma_{+}$ in \Cref{PathRootsR} and \Cref{PathGamma}.
\begin{figure}[!h]
	\centering
	\includegraphics[scale=.43]{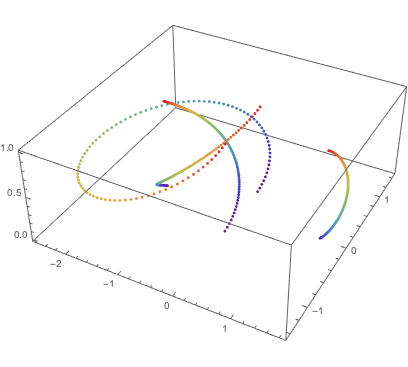}
	\includegraphics[scale=.4]{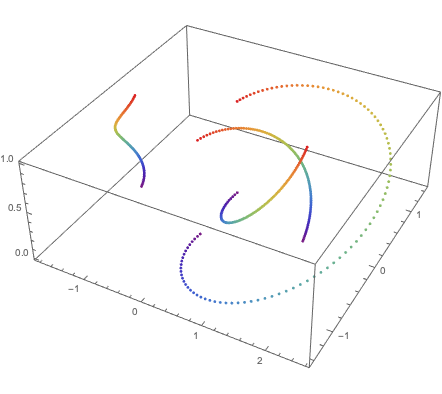}
	\caption{Plots of $\mathcal{X}_{N}\Par{\gamma_{-}}$ (left) and $\mathcal{X}_{N}\Par{\gamma_{+}}$ (right) for $N=100$.}
	\label{PathRootsR}
\end{figure}

\begin{figure}[!h]
	\centering
	\includegraphics[scale=.43]{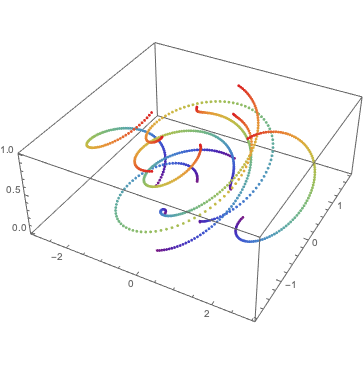}
	\includegraphics[scale=.6]{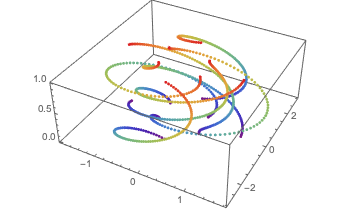}
	\caption{Plots of $\mathcal{Y}_{N}\Par{\gamma_{-}}$ (left) and $\mathcal{Y}_{N}\Par{\gamma_{+}}$ (right) for $N=100$.}
	\label{PathGamma}
\end{figure}
The permutations on the roots of $R_{0}$ induced by $\gamma_{-}$ and $\gamma_{+}$ are given as in \Cref{XGamma}, and the permutations on $y$-coordinates of the inflection points of $f_{0}=f$ induced by $\gamma_{-}$ and $\gamma_{+}$ are given as in \Cref{YGamma}.

\begin{figure}[!h]
	\centering 
	\begin{tikzpicture}
	\draw[->] (-3.5, 0) -- (3.5, 0) coordinate (x axis);
	\draw[->] (0, -1.5) -- (0, 1.5) coordinate (y axis);
	\filldraw (2*1.4678, 0) circle (2pt) node(1)[above]{$a$};
	\filldraw (2*-1.4678, 0) circle (2pt) node(2)[above]{$-a$};
	\filldraw (0, 2*0.3933) circle (2pt) node(3)[right]{$ia\frac{\Par{\sqrt{3}-1}^{2}}{2}$};
	\filldraw (0, 2*-0.3933) circle (2pt) node(4)[right]{$-ia\frac{\Par{\sqrt{3}-1}^{2}}{2}$};
	\draw[->, thick, red] (2) to[out=90, in=180] (3);
	\draw[->, thick, red] (3) to[out=270, in=90] (4);
	\draw[->, thick, red] (4) to[out=180, in=0] (2);
	\end{tikzpicture}\quad\quad\quad
	\begin{tikzpicture}
	\draw[->] (-3.5, 0) -- (3.5, 0) coordinate (x axis);
	\draw[->] (0, -1.5) -- (0, 1.5) coordinate (y axis);
	\filldraw (2*1.4678, 0) circle (2pt) node(1)[above]{$a$};
	\filldraw (2*-1.4678, 0) circle (2pt) node(2)[above]{$-a$};
	\filldraw (0, 2*0.3933) circle (2pt) node(3)[right]{$ia\frac{\Par{\sqrt{3}-1}^{2}}{2}$};
	\filldraw (0, 2*-0.3933) circle (2pt) node(4)[right]{$-ia\frac{\Par{\sqrt{3}-1}^{2}}{2}$};
	\draw[->, thick, blue] (1) to[out=270, in=0] (4);
	\draw[->, thick, blue] (4) to[out=90, in=270] (3);
	\draw[->, thick, blue] (3) to[out=0, in=180] (1);
	\end{tikzpicture}
	\caption{Permutation of the roots of $R_{0}$ induced by $\gamma_{-}$ (left) and $\gamma_{+}$ (right)}
	\label{XGamma}
\end{figure}
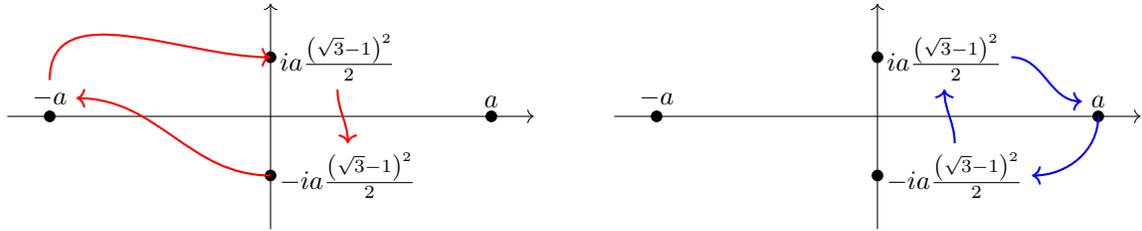

\begin{figure}[!h]
	\centering 
	\begin{tikzpicture}
	\draw[->] (-6, 0) -- (6, 0) coordinate (x axis);
	\draw[->] (0, -3.5) -- (0, 3.5) coordinate (y axis);
	\filldraw (2.5*1.3019, 0) circle (2pt) node(1)[above]{$b$};
	\filldraw (2.5*-1.3019, 0) circle (2pt) node(2)[above]{$-b$};
	\filldraw (0, 2.5*1.3019) circle (2pt) node(3)[right]{$ib$};
	\filldraw (0, 2.5*-1.3019) circle (2pt) node(4)[right]{$-ib$};
	\filldraw (2.5*0.6739, 2.5*0.6739) circle (2pt) node(5)[right]{$b\Par{1+i}\Par{\sqrt{3}-1}$};
	\filldraw (2.5*-0.6739, 2.5*-0.6739) circle (2pt) node(6)[left]{$-b\Par{1+i}\Par{\sqrt{3}-1}$};
	\filldraw (2.5*0.6739, 2.5*-0.6739) circle (2pt) node(7)[right]{$b\Par{1-i}\Par{\sqrt{3}-1}$};
	\filldraw (2.5*-0.6739, 2.5*0.6739) circle (2pt) node(8)[left]{$-b\Par{1-i}\Par{\sqrt{3}-1}$};
	\draw[->, thick, red] (4) to[out=90, in=330] (8);
	\draw[->, thick, red] (8) to[out=0, in=180] (5);
	\draw[->, thick, red] (5) to[out=210, in=60] (4);
	\draw[->, thick, red] (3) to[out=270, in=150] (7);
	\draw[->, thick, red] (7) to[out=180, in=0] (6);
	\draw[->, thick, red] (6) to[out=30, in=240] (3);
	\end{tikzpicture}
	\begin{tikzpicture}
	\draw[->] (-6, 0) -- (6, 0) coordinate (x axis);
	\draw[->] (0, -3.5) -- (0, 3.5) coordinate (y axis);
	\filldraw (2.5*1.3019, 0) circle (2pt) node(1)[above]{$b$};
	\filldraw (2.5*-1.3019, 0) circle (2pt) node(2)[above]{$-b$};
	\filldraw (0, 2.5*1.3019) circle (2pt) node(3)[right]{$ib$};
	\filldraw (0, 2.5*-1.3019) circle (2pt) node(4)[right]{$-ib$};
	\filldraw (2.5*0.6739, 2.5*0.6739) circle (2pt) node(5)[right]{$b\Par{1+i}\Par{\sqrt{3}-1}$};
	\filldraw (2.5*-0.6739, 2.5*-0.6739) circle (2pt) node(6)[left]{$-b\Par{1+i}\Par{\sqrt{3}-1}$};
	\filldraw (2.5*0.6739, 2.5*-0.6739) circle (2pt) node(7)[right]{$b\Par{1-i}\Par{\sqrt{3}-1}$};
	\filldraw (2.5*-0.6739, 2.5*0.6739) circle (2pt) node(8)[left]{$-b\Par{1-i}\Par{\sqrt{3}-1}$};
	\draw[->, thick, blue] (1) to[out=180, in=0] (6);
	\draw[->, thick, blue] (6) to[out=30, in=330] (8);
	\draw[->, thick, blue] (8) to[out=0, in=0] (1);
	\draw[->, thick, blue] (2) to[out=0, in=180] (5);
	\draw[->, thick, blue] (5) to[out=210, in=150] (7);
	\draw[->, thick, blue] (7) to[out=180, in=180] (2);
	\end{tikzpicture}
	\caption{Permutation on the $y$-coordinates of inflection points of $f$ induced by $\gamma_{-}$ (above) and $\gamma_{+}$ (below).}
	\label{YGamma}
\end{figure}
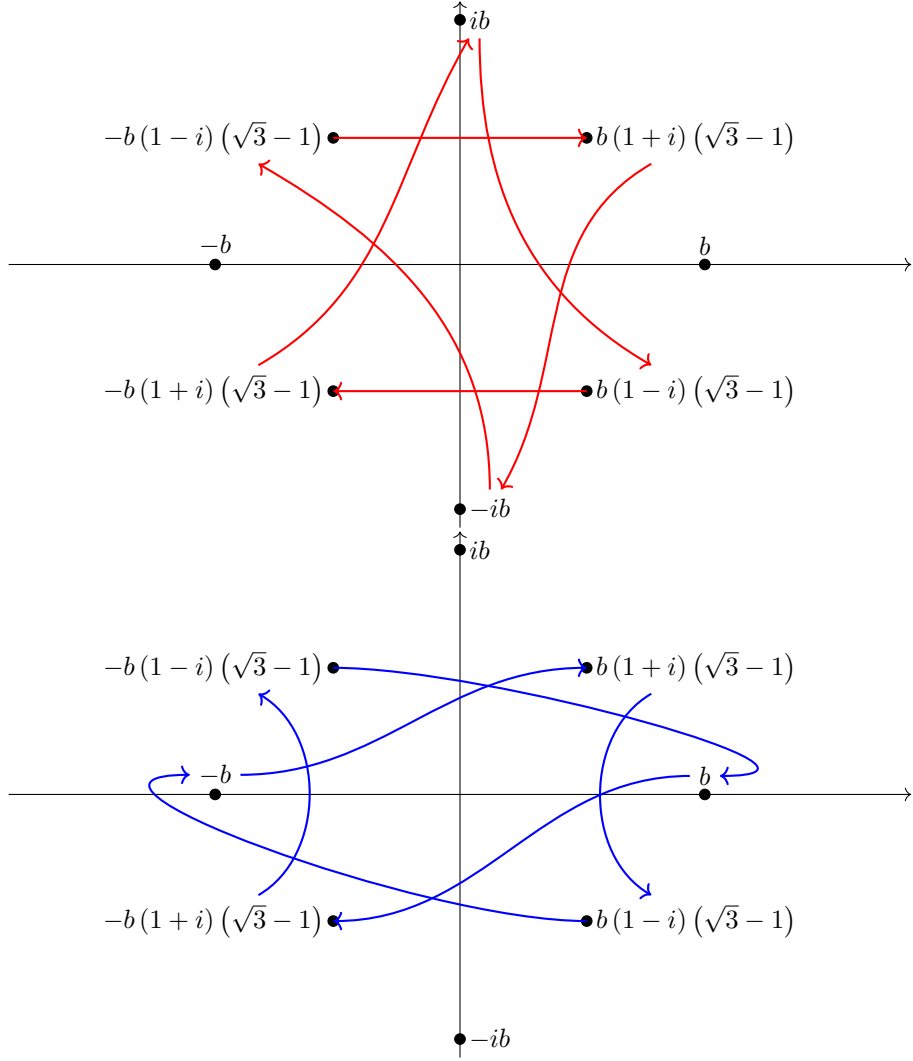

\begin{enumerate}
	\item The loop $\gamma_{-}$ induces a permutation $G_{1}\in W\Par{E_{6}}$ which maps
	\begin{align*}
	L_{1}&\mapsto L_{1} &
	L_{2}&\mapsto L_{2} & 
	L_{3}&\mapsto L_{4, 5} \\
	L_{4}&\mapsto L_{6} &
	L_{5}&\mapsto L_{3, 4} &
	L_{6}&\mapsto L_{3, 5}
	\end{align*}
	
	\item The loop $\gamma_{+}$ induces a permutation $G_{2}\in W\Par{E_{6}}$ which maps
	\begin{align*}
	L_{1}&\mapsto L_{5, 6} &
	L_{2}&\mapsto L_{1, 6} & 
	L_{3}&\mapsto L_{3} \\
	L_{4}&\mapsto L_{4} &
	L_{5}&\mapsto L_{2} &
	L_{6}&\mapsto L_{1, 5}
	\end{align*}
\end{enumerate}
and therefore, with respect to our basis $\WBPar{e_{0}, \BPar{L_{1}}, \ldots, \BPar{L_{6}}}$ we obtain the matrices
\[G_{1}=\begin{pmatrix*}[r]
2 & 0 & 0 & 1 & 0 & 1 & 1 \\
0 & 1 & 0 & 0 & 0 & 0 & 0 \\
0 & 0 & 1 & 0 & 0 & 0 & 0 \\
-1 & 0 & 0 & 0 & 0 & -1 & -1 \\
-1 & 0 & 0 & -1 & 0 & -1 & 0 \\
-1 & 0 & 0 & -1 & 0 & 0 & -1 \\
0 & 0 & 0 & 0 & 1 & 0 & 0
\end{pmatrix*}
\quad\text{and}\quad 
G_{2}=\begin{pmatrix*}[r]
2 & 1 & 1 & 0 & 0 & 0 & 1 \\
-1 & 0 & -1 & 0 & 0 & 0 & -1 \\
0 & 0 & 0 & 0 & 0 & 1 & 0 \\
0 & 0 & 0 & 1 & 0 & 0 & 0 \\
0 & 0 & 0 & 0 & 1 & 0 & 0 \\
-1 & -1 & 0 & 0 & 0 & 0 & -1 \\
-1 & -1 & -1 & 0 & 0 & 0 & 0
\end{pmatrix*}\]
where
\[\mathbf{G}:=\Bracket{G_{1}, G_{2}}<\rho\Par{SL_{2}\Par{\zz}}<W\Par{E_{6}}.\]
With the help of \cite{Mathematica}, we compute the order of the matrix group generated by $G_{1}$ and $G_{2}$, and we obtain $\Abs{\mathbf{G}}=24=2^{3}\cdot 3$. Moreover, we make the following observations about $\mathbf{G}$:
\begin{itemize}
	\item $\BPar{G_{1}, G_{2}}\neq 1$, so $\mathbf{G}$ is not abelian.
	\item $G_{1}^{3}=G_{2}^{3}=1$ and $\Bracket{G_{1}}\neq \Bracket{G_{2}}$, so $3$-Sylow subgroups of $\mathbf{G}$ are not normal.
	\item $\text{ord}\Par{G_{1}G_{2}}=6$ and $\text{ord}\Par{\BPar{G_{1}, G_{2}}}=4$, so $\mathbf{G}$ contains elements of order $4$ and $6$.
\end{itemize}
Using the classification of groups of order $24$ in \cite{Grp24}, the first $3$ observations reduce our posibilities for $\mathbf{G}$ to $SL_{2}\Par{\zz/3\zz}$, $S_{4}$ and $A_{4}\times \zz/2\zz$. Since $S_{4}$ has no elements of order $6$ and $A_{4}\times \zz/2\zz$ has no elements of order $4$, we conclude $\mathbf{G}\cong SL_{2}\Par{\zz/3\zz}$.

\section{Proof of the Main Theorem}\label{MainProofSection}
Now that we have computed explicitly $H_{1}, H_{2}, G_{1}, G_{2}$ inside $\text{Im}\Par{\rho}$, we are ready to conclude. Recall the groups $\mathbf{H}, \mathbf{G}<\text{Im}\Par{\rho}$ are defined in \Cref{ComputingImagesSection} and we have shown
\[\mathbf{H}:=\Bracket{H_{1}, H_{2}}\cong \mathcal{H}_{3}\Par{\zz/3\zz}\quad\text{ and }\quad\mathbf{G}:=\Bracket{G_{1}, G_{2}}\cong SL_{2}\Par{\zz/3\zz}.\]
Consider the subgroup of $\text{Im}\Par{\rho}$ generated by $H_{1}, H_{2}, G_{1}, G_{2}$, \[\mathbf{I}:=\Bracket{H_{1}, H_{2}, G_{1}, G_{2}}.\]
We will show $\mathbf{I}\cong \text{Im}\Par{\rho}$ in order to conclude.

\begin{proof}[Proof of \Cref{MainTheorem}]
	 With the help of \cite{Mathematica}, we compute the intersection of the matrix groups $\mathbf{H}$ and $\mathbf{G}$ and obtain that $\mathbf{H}\cap \mathbf{G}=\WBPar{1}$. Conjugating $H_{1}$ and $H_{2}$ by $G_{1}$ and $G_{2}$ we obtain
	\begin{align*}
		G_{1}H_{1}G_{1}^{-1}&=\Omega H_{1}H_{2} &  G_{2}H_{1}G_{2}^{-1}&=H_{1} \\
		G_{1}H_{2}G_{1}^{-1}&=H_{2} & G_{2}H_{2}G_{2}^{-1}&=\Omega^{-1}H_{1}^{-1}H_{2}.
	\end{align*}
	Therefore, $\mathbf{H}$ is normal in $\mathbf{I}$. These two facts together imply that
	\[\mathbf{I}\cong \mathbf{H}\rtimes \mathbf{G}\cong \mathcal{H}_{3}\Par{\zz/3\zz}\rtimes SL_{2}\Par{\zz/3\zz}.\]
	The homomorphism $\varphi:SL_{2}\Par{\zz/3\zz}\to \text{Aut}\Par{\mathcal{H}_{3}\Par{\zz/3\zz}}$ described in \Cref{RestrictingImageSection} completely describes this semidirect product structure, and an explicit isormorphism $\mathbf{I}\to \mathcal{H}_{3}\Par{\zz/3\zz}\rtimes_{\varphi} SL_{2}\Par{\zz/3\zz}$ is given by
\begin{align*}
H_{1}&\mapsto \Par{\begin{pmatrix}
	1 & 1 & 1 \\
	0 & 1 & 0 \\
	0 & 0 & 1
	\end{pmatrix}, I}\quad
G_{1}\mapsto \Par{I, \begin{pmatrix}
	1 & 0 \\
	1 & 1
	\end{pmatrix}} \\
H_{2}&\mapsto \Par{\begin{pmatrix}
	1 & 0 & 1 \\
	0 & 1 & 1 \\
	0 & 0 & 1
	\end{pmatrix}, I}\quad
G_{2}\mapsto \Par{I, \begin{pmatrix}
	1 & 2 \\
	0 & 1
	\end{pmatrix}}.
\end{align*}
Since $\Abs{\mathbf{I}}\leq \Abs{\text{Im}\Par{\rho}}\leq\Abs{C_{W\Par{E_{6}}}\Par{\Omega}}=648$ and
\[\Abs{\mathbf{I}}=\Abs{\mathbf{H}}\cdot \Abs{\mathbf{G}}=\Abs{\mathcal{H}_{3}\Par{\zz/3\zz}}\cdot \Abs{SL_{2}\Par{\zz/3\zz}}=27\cdot 24=648\]
we have that the equalities must hold. By \Cref{ImageInCentralizer}, we have that $\text{Im}\Par{\rho}\subset C_{W\Par{E_{6}}}\Par{\Omega}$, so we conclude
\[\mathbf{I}\cong \text{Im}\Par{\rho}\cong C_{W\Par{E_{6}}}\Par{\Omega}\cong \mathcal{H}_{3}\Par{\zz/3\zz}\rtimes_{\varphi} SL_{2}\Par{\zz/3\zz}.\]
\end{proof}

%\cite{dolglib, Grp24, morita, harris, hartshorne, frame}
\printbibliography
\vspace{.5cm}
\noindent Department of Mathematics, University of Chicago\\
E-mail: \href{mailto:amedrano@math.uchicago.edu}{amedrano@math.uchicago.edu}
\end{document}